\documentclass[11pt,a4paper]{amsart} %Falls ein Draft erwuenscht ist, [draft] anstelle der speziefischen angaben einfuegen. overfull \hbox ersichtlich,   
\usepackage{graphicx} % ... if you want to include encapsulated postscript figures
\usepackage{amssymb, latexsym,amsmath} % ... if you need lots of math symbols
\usepackage{amsthm}                              %for proof-environments

\usepackage{textcomp}
\usepackage[applemac]{inputenc}% ... umlauts
\usepackage[numbers]{natbib}%... improved citations and bibliography
\usepackage[dvips=true,bookmarks=true]{hyperref} % ... only needed for PDF generation to provide clickable urls
\usepackage{verbatim} % um laengere passagen auszukommentieren mit begin{comment}

\usepackage{pb-diagram}

\usepackage[all]{xy}

\usepackage{epsfig}
\usepackage[dvips]{color}

\DeclareGraphicsExtensions{.pdf,.eps}
\usepackage{psfrag} %to write LateX text into the eps figures
\newtheorem{thm}{Theorem}[section]
\newtheorem{prop}[thm]{Proposition}

\newtheorem{lem}[thm]{Lemma}

\theoremstyle{definition}
\newtheorem{def.}[thm]{Definition}

\newtheorem*{acknowledgments}{Acknowledgments}
\newtheorem{rem}[thm]{Remark}

\newcommand\congto{\overset{\cong}{\rightarrow }}
\renewcommand\int{\operatorname{int}}
\newcommand \C {\mathbb{C}}
\newcommand \Z {\mathbb{Z}}
\newcommand \id {\operatorname{id}}

\title[Kirby calculus for null-homotopic framed links]{On Kirby
  calculus for null-homotopic framed links in $3$-manifolds}

\author[Kazuo Habiro]{Kazuo Habiro}
\address{Research Institute for Mathematical Sciences, Kyoto
University, Kyoto 606-8502, Japan}
\email{habiro@kurims.kyoto-u.ac.jp}

\author[Tamara Widmer]{Tamara Widmer}
\address{Universit\"at Z\"urich, Winterthurerstr. 190
CH-8057 Z\"urich, Switzerland}
\email{tamara.widmer@math.uzh.ch}

\date{June 20, 2013 (First version: January 30, 2013)}

\begin{document}

\begin{abstract}
  Kirby proved that two framed links in $S^3$ give
  orientation-preserving homeomorphic results of surgery if and only
  if these two links are related by a sequence of two kinds of moves
  called stabilizations and handle-slides.  Fenn and Rourke gave a
  necessary and sufficient condition for two framed links in a closed,
  oriented $3$-manifold to be related by a finite sequence of these
  moves.

  The purpose of this paper is twofold. We first give a generalization
  of Fenn and Rourke's result to $3$-manifolds with boundary.  Then we
  apply this result to the case of framed links whose components are
  null-homotopic in the $3$-manifold.
\end{abstract}

\maketitle

%%%%%%%%%%%%%%%%%%%%%%%%%%%%%%%%%%%%%%%%%%%%%%%%%%%%%%%%%%%%%%%%%%%%%%%%%%%%%%%%%%%%%%%%

\section{Introduction}

In 1978, Kirby \cite{Kirby:1978} proved that two framed links in $S^3$
have homeomorphic result of surgery if and only if they are related by
a sequence of two kinds of moves called stabilizations and
handle-slides.  This result enables one to construct a $3$-manifold
invariant by constructing a link invariant which is invariant under
these moves.  Fenn and Rourke \cite{FennRourke:1978} generalized
Kirby's theorem to framed links in closed $3$-manifolds, and
Roberts \cite{Roberts:1997} generalized it to framed links in
$3$-manifolds with boundary.

Fenn and Rourke \cite{FennRourke:1978} also considered the equivalence
relation on framed links in an arbitrary closed, oriented
$3$-manifold generated by stabilizations and handle-slides. Here we
state Fenn and Rourke's theorem, leaving some details to the original
paper \cite{FennRourke:1978}. Let $M$ be a closed, oriented
$3$-manifold.  For a framed link $L$ in $M$, we will denote by $W_L$
the $4$-manifold obtained from $M\times I$ by attaching $2$-handles
along $L\times\{1\}\subset\partial(M\times I)$ in a way determined by
the framing. Note that $W_L$ is a
cobordism between $M$ and $M_L$, where $M_L$ denotes the $3$-manifold
obtained from $M$ by surgery along $L$. The inclusions
$M_L\hookrightarrow W_L \hookleftarrow M$ induce surjective
homomorphisms
\begin{gather*}
  \pi_1(M_L)\twoheadrightarrow \pi_1(W_L)\twoheadleftarrow \pi_1(M).
\end{gather*}
The kernel of the homomorphism $\pi_1(M)\rightarrow \pi_1(W_L)$ is normally generated
by the homotopy classes of components of $L$.

\begin{thm}[Fenn--Rourke \cite{FennRourke:1978}]\label{FR}
Let $M$ be a closed, oriented $3$-manifold, and let $L$ and $L'$ be
two framed links in $M$.  Then $L$ and $L'$ are related by a sequence
of stabilizations and handle-slides if and only if there exist an
orientation-preserving homeomorphism $h\colon M_{L}\to M_{L'}$ and an
isomorphism
\begin{gather*}
f\colon  \pi_1(W_{L})\rightarrow \pi_1(W_{L'}) ,
\end{gather*}
such that the diagram
\begin{equation}
 \divide\dgARROWLENGTH by2
\begin{diagram} \label{diagram:delta}
\node{\pi_1(M_{L})}     \arrow[2]{e,t}{h_*}\arrow{s,l}{}
\node[2]{\pi_1(M_{L'})} \arrow{s,r}{}                   \\
\node{\pi_1(W_{L})}   \arrow[2]{e,t}{f}
\node[2]{\pi_1(W_{L'})}    \\
\node[2]{\pi_1(M)} \arrow{nw,b}{}\arrow{ne,r}{}
\end{diagram}
\end{equation}
commutes and we have $\rho_*([W])=0\in H_4(\pi_1(W_L),\Z)$.  Here
\begin{itemize}
\item $W$ is the closed $4$-manifold obtained from $W_L$ and $W_{L'}$ by gluing along their boundaries
using $\id_M$ and $h$,
\item $[W]\in H_4(W,\Z)$ is the fundamental class, and
\item $\rho_*\colon H_4(W,\Z)\to H_4(\pi_1(W_L),\Z)$ is induced by a map
  $\rho\colon W\to K(\pi_1(W_L),1)$ obtained by gluing natural maps from
  $W_L$ and $W_{L'}$ to $K(\pi_1(W_L),1)$. 
\end{itemize}
See \cite{FennRourke:1978} for more details.
\end{thm}

One of the main results of the present paper, Theorem \ref{thm:FR-generalization},
is a generalization of Theorem \ref{FR} to $3$-manifolds with
boundary.  (A generalization of Theorem \ref{FR} to
$3$-manifolds with boundary has been stated in
\cite{Garoufalidis-Kricker:2002}, but unfortunately the statement in \cite{Garoufalidis-Kricker:2002} is not
correct for $3$-manifolds with more than one boundary components.)

An obstruction to making Theorems \ref{FR} and
\ref{thm:FR-generalization} useful is the homological condition
$\rho_*([W])=0$.  Given framed links $L,L'$ in $M$ as in Theorems \ref{FR}
and \ref{thm:FR-generalization}, it is not always easy to see whether
we have $\rho_*([W])=0$ or not. 
However, if $H_4(\pi_1(W_L),\Z)=0$, then clearly we have $\rho_*([W])=0$.  

A large class of groups with vanishing $H_4(-,\Z)$ is the $3$-manifold
groups.  It seems to have been well known for a long time that if $M$
is a compact, connected, oriented $3$-manifold, then we have
$H_4(\pi_1(M),\Z)=0$ (see Lemma \ref{r3}).  So, if the components of
the framed links $L$ and $L'$ in $M$ are null-homotopic, then since
$\pi_1(W_L)\cong\pi_1(M)$ is a $3$-manifold group, we have
$H_4(\pi_1(W_L),\Z)=0$ and $\rho_*([W])=0$.  Thus, for null-homotopic framed
links, we do not need the condition $\rho_*([W])=0$, see Theorem
\ref{r4}.

Cochran, Gerges and Orr \cite{Cochran-Gerges-Orr} studied surgery
along null-homologous framed links with diagonal linking matrices with
diagonal entries $\pm1$, and also surgery along more special classes
of framed links.  This includes null-homotopic framed links with
diagonal linking matrices with diagonal entries $\pm1$.  Let us call
such a framed link {\em $\pi_1$-admissible}.  Surgery along a
$\pi_1$-admissible framed link $L$ in a $3$-manifold $M$ gives a
manifold $M_L$ whose fundamental group is ``very close'' to that of
$M$.  In \cite{Cochran-Gerges-Orr} it is proved that, for all $d\ge1$,
we have $\pi_1(M_L)/\Gamma_d\pi_1(M_L)\cong\pi_1(M)/\Gamma_d\pi_1(M)$,
where for a group $G$, $\Gamma_dG$ denotes the $d$th lower central
series subgroup of $G$. 

For $\pi_1$-admissible framed links in a $3$-manifold, we can combine
Theorem~\ref{r4} with Proposition \ref{prop:6} proved by the first
author \cite{Habiro:2006} to obtain a refined version of Theorem
\ref{r4}, see Theorem \ref{prop:admissible}.  This theorem gives a
necessary and sufficient condition for two $\pi_1$-admissible framed
links in $M$ to be related by a sequence of stabilizations and {\em
  band-slides} \cite{Habiro:2006}, which are pairs of algebraically
cancelling handle-slides, see Section \ref{section:admissible}.

We apply Theorem  \ref{prop:admissible} to surgery along
null-homotopic framed links in cylinders over surfaces.  Surgery along
a $\pi_1$-admissible framed link in a cylinder over a surface gives a
homology cylinder of a special kind.

The organization of the rest of the paper is as follows.  In Section
\ref{sec:gener-fenn-rourke-1}, we introduce some notations and
preliminary facts, and then state and prove the generalization of Fenn
and Rourke's theorem to $3$-manifolds with boundary.  In
Section~\ref{section: application1}, we focus on the case of
null-homotopic framed links. In Section~\ref{section:admissible}, we
consider $\pi_1$-admissible framed links.  In Section
\ref{sec:example}, we give an example which illustrates the conditions
needed in Theorem~\ref{thm:FR-generalization}.

\begin{acknowledgments}
  The first author was partially supported by JSPS, Grant-in-Aid for
  Scientific Research (C) 24540077.  The second author was supported
  by SNF, No. 200020\_134774/1.
  
  The authors thank Anna Beliakova for encouraging conversations.
\end{acknowledgments}

\section{Generalization of Fenn and Rourke's Theorem}
\label{sec:gener-fenn-rourke-1}

In this section we state and prove a generalization of Theorem
\ref{FR} to $3$-manifolds with nonempty boundary.  We start by giving
necessary notations which are used throughout this paper. Then we
introduce the conditions under which Theorem \ref{FR} holds for
manifolds with boundary and give the statement and the proof of our
generalization of Theorem \ref{FR}.  Our construction mainly follows
\cite{FennRourke:1978} and borrows some ideas also from
\cite{Garoufalidis-Kricker:2002}.

Let $M$ be a compact, connected, oriented $3$-manifold, possibly with
nonempty boundary.

A {\em framed link} $L=L_1\cup\dots\cup L_l$ in $M$ is a link (i.e.,
disjoint union of finitely many embedded circles in $M$) such that
each component $L_i$ of $L$ is given a framing, i.e., a homotopy class
of trivializations of the normal bundle.  Such a framing of $L_i$ may
be given as a homotopy class of a simple closed curve $\gamma_i$ in the
boundary $\partial N(L_i)$ of a tubular neighborhood $N(L_i)$ of
$L_i$ in $M$ which is homotopic to $L_i$ in $N(L_i)$.

For a framed link $L\subset M$ as above, let $M_L$ denote the result
from $M$ of surgery along $L$.  This manifold is obtained from $M$ by
removing the interiors of $N(L_i)$, and gluing a solid torus
$D^2\times S^1$ to $\partial N(L_i)$ so that the curve $\partial
D^2\times \{*\}$, $*\in S^1$, is attached to $\gamma_i\subset\partial N(L_i)$ for
each $i=1,\ldots,l$.

Surgery along a framed link can be defined by using $4$-manifolds as
well.  Let $L$ be a framed link in $M$.  Let $W_L$ denote the
$4$-manifold obtained from the cylinder $M \times I$ by attaching a
$2$-handle $h_i\cong D^2\times D^2$ along $N(L_i)\times\{1\}$ using
the a homeomorphism
\begin{gather*}
  S^1\times D^2\overset{\cong}{\to} N(L_i),
\end{gather*}
which maps $S^1\times\{*\}$, $*\in\partial D^2$, onto the framing $\gamma_i$.
We have a natural identification
$$\partial W_L\cong M\underset{\partial M}{\cup}
(\partial M\times I)\underset{\partial M_L}{\cup} M_L,$$
Thus, $W_L$ is a cobordism between $M$ and $M_L$. 
Note that $\partial W_L$ is connected if $\partial M\neq\emptyset$.

We define two moves on framed links. A {\em handle-slide} replaces one
component $L_i$ of $L$ with a band sum $L_i'$ of $L_i$ and a parallel copy
of another component $L_j$ as in Figure 
\ref{Fhandleslide}, where the blackboard framing convention is used. 
A {\em stabilization} adds to or removes from a link $L$ an isolated
$\pm1$-framed unknot.

\begin{figure}[t]
    \begin{center}\input{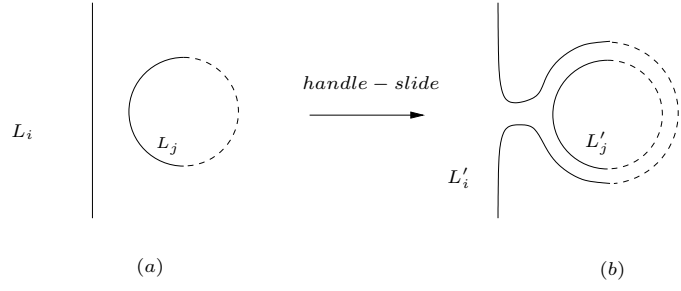}\end{center}
    \caption{(a) Two components $L_i$ and $L_j$ of a framed link.  (b) The result of a handle-slide of $L_i$ over $L_j$.}
    \label{Fhandleslide}
  \end{figure}

\subsection{Some notations}
\label{sec:some-notations-1}
We introduce some notations which we need in the statement of our
generalization of Theorem \ref{FR}, and which will be used in later
sections as well.

Let $M$ be a compact, connected, oriented $3$-manifold with {\em nonempty}
boundary.

  Let $F_1, \ldots, F_n$ ($n\ge1$) denote the components of $\partial
M$.  For each $k=1,\ldots,n$, choose a base point $p_k \in F_k$.  We
denote by $\pi_1(M;p_1,p_k)$ the set of homotopy classes of paths from
$p_1$ to $p_k$ in $M$.  We consider $p_1$ as the base point of $M$,
and write
\begin{gather*}
  \pi_1(M)=\pi_1(M;p_1)=\pi_1(M;p_1,p_1).
\end{gather*}

Let $L$ be a framed link in $M$ as before.  
We consider the $4$-manifold $W_L$ defined in Section
\ref{sec:gener-fenn-rourke-1}.
For $k=1,\dots,n$, set $p_k^L=p_k\times \{1\}\in\partial M_L$ and
$\gamma_k=p_k\times I \subset \partial W_L$.  Note that $\gamma _k$ is
an arc in $\partial W$ from $p_k\in\partial M\subset \partial W_L$ to
$p_k^L$.

  The inclusions
\begin{equation*}  M \overset{i}{\hookrightarrow} W_L  \overset{i'}{\hookleftarrow} M_L \end{equation*}
 induce surjective maps 
$$\pi_1(M;p_1,p_k)\overset{i_k}{\longrightarrow} \pi_1(W_L;p_1,p_k) \overset{i'_k}{\longleftarrow} \pi_1(M_L;p^L_1,p^L_k)$$
for $k=1,\ldots,n$.  Here $i'_k$ is defined to be the composition
$$\pi_1(M_L;p^L_1,p^L_k)\overset{i'_k}{\longrightarrow} \pi_1(W_L;p^L_1,p^L_k)\underset{\gamma_1,\gamma_k}{\cong} \pi_1(W_L;p_1,p_k),$$
where the second isomorphism is induced by the arcs $\gamma_1$ and $\gamma_k$.

We regard $p_1^L$ as the base point of $M_L$  and write
$\pi_1(M_L):=\pi_1(M_L;p_1^L)$.
The point $p_1$ is regarded as a base point of $W_L$ as well as of $M$, and we
set $\pi_1(W_L):=\pi_1(W_L;p_1)$. 

An Eilenberg--Mac Lane space $K(\pi_1(W_L),1)$ can be obtained from
$W_L$ by attaching cells which kill higher homotopy groups. Thus, there
is a natural inclusion 
\begin{gather*}
  \rho_L\colon  W_L \hookrightarrow K(\pi_1(W_L),1).
\end{gather*}

\subsection{Construction of a homology class}
\label{sec:constr-homol-class}

Now, consider two framed links $L$ and $L'$ in $M$, and suppose that
there exists a homeomorphism $h\colon  M_L\to M_{L'}$ relative to the
boundary.  Moreover, we assume that there exist isomorphisms
$f_k\colon \pi_1(W_{L};p_1,p_k) \to \pi_1(W_{L'};p_1,p_k)$ such that the
diagram
\begin{equation}
 \divide\dgARROWLENGTH by2
\begin{diagram} \label{diagram:delta-enhanced}
\node{\pi_1(M_{L};p_1^L,p_k^L)}       %%% row 1
\arrow[2]{e,t}{h_k}\arrow{s,l}{i'_k}
\node[2]{\pi_1(M_{L'};p_1^{L'},p_k^{L'})}
\arrow{s,r}{i'_k}                   \\
\node{\pi_1(W_{L};p_1,p_k)}       %%% row 2
\arrow[2]{e,t}{f_k}
\node[2]{\pi_1(W_{L'};p_1,p_k)}    \\
\node[2]{\pi_1(M;p_1,p_k)}
\arrow{nw,b}{i_k}\arrow{ne,r}{i_k}
\end{diagram}
\end{equation}
commutes for $k=1,\ldots, n$.  For $k=2,\ldots,n$, that ``$f_k$ is an
isomorphism'' means that $f_k$ is a bijection.  (Here, if $f_k$ is a
bijection which makes the above diagram commutes, then it follows that
$f_k$ is an isomorphism between the $\pi _1(W_L)$-set
$\pi_1(W_L;p_1,p_k)$ and the $\pi _1(W_{L'})$-set
$\pi_1(W_{L'};p_1,p_k)$ along the group isomorphism
$f_1\colon\pi_1(W_L)\to\pi _1(W_{L'})$.)

In the following, we define a homology class
\begin{gather*}
  \rho_*([W])\in H_4(\pi_1(W_L),\Z),
\end{gather*}
by constructing a closed $4$-manifold $W$ and a map $\rho\colon
W\rightarrow K(\pi_1(W_L) ,1)$.  

As in \cite{Garoufalidis-Kricker:2002},
define a $4$-manifold $W$  by
\begin{gather*}
  W:=W_L\cup_\partial (-W_{L'}),  
\end{gather*}
where we glue $W_L$ and $-W_{L'}$ (the orientation reversal of
$W_{L'}$) along the boundaries using the
identity map on $M\cup(\partial M \times I)$ and the homeomorphism
$h\colon M_L\overset{\cong}{\rightarrow} M_{L'}$.

Consider the following diagram
\begin{gather}
  \label{e1}
  \xymatrix{
    \partial W_L \ar[r]^{u'}\ar[d]_{u} & W_{L'} \ar[d]_{j'}\ar[rdd]^{\tilde{\rho}_{L'}}& \\
    W_L\ar[r]^{j}\ar[rrd]_{\rho_L}&W\ar[dr]^{\rho }&\\
    &&K(\pi_1(W_L),1),
  }
\end{gather}
where $u,u',j,j'$ are inclusions.    The map $\tilde\rho_{L'}\colon W_{L'}\to
K(\pi_1(W_L),1)$ is the composite
\begin{gather*}
  W_{L'}\overset{\rho _{L'}}{\rightarrow} K(\pi_1(W_{L'}),1)\overset{K(f_1^{-1},1)}{\underset{\simeq}{\rightarrow}}
  K(\pi_1(W_L),1).
\end{gather*}
Here $K(f_1^{-1},1)$ is a homotopy equivalence, unique up to homotopy.
By the definition of $W$, the square is a pushout.  Hence, to prove
existence of $\rho $ such that $\rho j=\rho _L$ and $\rho
j'=\tilde{\rho}_{L'}$, we need only to show that 
$\rho _L
u\simeq\tilde\rho_{L'}u'$, 
which easily follows from Lemma \ref{r1} below.  (Proof of this lemma is the
place where commutativity of \eqref{diagram:delta-enhanced} in
Theorem~\ref{thm:FR-generalization} is necessary not only for $k=1$
but also for $k=2,\ldots,n$.)

\begin{lem}
  \label{r1}
  Under the above situation, the following diagram commutes.
\begin{equation}\label{diagram:boundary-commutes}
\begin{diagram}
\node{\pi_1(\partial W_L)}    
\arrow{e,t}{u'_\ast}\arrow{s,l}{u_\ast}
\node{\pi_1(W_{L'})}
\arrow{s,r}{j'_\ast}                 \\
\node{\pi_1(W_{L})} 
\arrow{ne,l} {f_1}
\arrow{e,t}{j_\ast}
\node{\pi_1(W)} 
\end{diagram}
\end{equation}
\end{lem}

\begin{proof}
  Since $u_*$ is surjective and the square is commutative,
  $u'_*=f_1u_*$ implies $j_*=j'_*f_1$.

  Let us prove $u'_*=f_1u_*$.
  For $k=2,\ldots,n$, choose an arc $c_k$ in $M$ from $p_1$ to $p_k$ disjoint from $L$. Set 
$$d_k=(c_k\times\{0,1\}) \cup (\partial c_k\times I),$$ 
which is a loop in $\partial W_L$ based at $p_1$.
 The fundamental group $\pi_1(\partial W_L)$ is then generated by the elements $d_2,\ldots, d_n$ and the images of the maps $i_\ast\colon \pi_1(M)\to \pi_1(\partial W_L)$ and $i'_\ast\colon \pi_1(M_L)\to \pi_1(\partial W_L)$. Hence $u_\ast'=f_1 u_\ast$ is reduced to the following:
\begin{itemize}
 \item[(a)] $u_\ast' i_\ast=f_1 u_\ast i_\ast\colon \pi_1(M)\to \pi_1(W_{L'})$,
 \item[(b)] $u_\ast' i'_\ast=f_1 u_\ast i'_\ast\colon \pi_1(M_L)\to \pi_1(W_{L'})$,
 \item[(c)] $u_\ast' (d_k)=f_1 u_\ast(d_k)$ for $k=2,\dots,n$.
\end{itemize}
(a) (resp. (b)) follows from commutativity of the lower (resp. upper)
part of diagram~\eqref{diagram:delta-enhanced} for $k=1$.  (c) follows
from commutativity of diagram~\eqref{diagram:delta-enhanced} for
$k=2,\ldots,n$.
\end{proof}

\subsection{Statement of the theorem}
\label{sec:statement-theorem}

Now we can state our generalization of Theorem \ref{FR} to $3$-manifolds
with boundary.

\begin{thm}\label{thm:FR-generalization} 
  Let $M$ be a compact, connected, oriented $3$-manifold with $n > 0$
  boundary components, and let $L,L'\subset M$ be framed links.
  Then the following conditions are equivalent.
  \begin{enumerate}
  \item $L$ and $L'$ are related by a sequence of stabilizations
    and handle-slides.
  \item There exist a homeomorphism $h\colon  M_{L}\to M_{L'}$ relative to
    the boundary and isomorphisms $f_k\colon \pi_1(W_{L};p_1,p_k)\rightarrow
    \pi_1(W_{L'};p_1,p_k)$ for $k=1,\ldots,n$ such that diagram
    \eqref{diagram:delta-enhanced} commutes for $k=1,\ldots,n$ and
    $\rho _*([W])=0 \in H_4(\pi_1(W_L))$.
  \end{enumerate}
\end{thm}

\begin{rem}
  Theorem \ref{FR} can be derived from the case $\partial M=S^2$ of Theorem
 \ref{thm:FR-generalization}.
\end{rem}

\begin{rem}
  In a paper in preparation \cite{HW}, we will give an example in which
  a nonzero homology class $\rho_*([W])$ is realized.
\end{rem}

\subsection{Proof of the theorem}
We need the following lemma which gives a necessary and sufficient
condition for $\rho _*([W]) \in H_4(\pi_1(W_L))$
to vanish.

\begin{lem}[{\cite[Lemma 9]{FennRourke:1978}, \cite[Lemma
	2.1]{Garoufalidis-Kricker:2002}}]
\label{lem:FR-Lemma9} 
In the situation of Theorem \ref{thm:FR-generalization}, 
we have  $\rho_*([W])=0$ if and only if
 the connected sum of $W$ with some copies of $\pm\C P^2$ is the
boundary of an oriented $5$-manifold $\Omega$ in such a way that the
diagram
\begin{equation}\label{diagram:FR-Lemma9}
\divide\dgARROWLENGTH by2
\begin{diagram}
\node{\pi_1(W_L;p_1)}
\arrow{se,b}{j_\ast}
\arrow[2]{e,t}{f_1}
\node[2]{\pi_1(W_{L'};p_1)}
\arrow{sw,r}{j_\ast '}\\
\node[2]{\pi_1(\Omega;p_1)}
\end{diagram}
\end{equation}
commutes and $j_\ast, j_\ast '$ are split injections induced by the inclusions $j\colon W_L \hookrightarrow \Omega$ and $j'\colon W_{L'}\hookrightarrow \Omega$.
\end{lem} 

\begin{proof}[Proof of Theorem \ref{thm:FR-generalization}]

The proof that (1) implies (2) is almost the same as the proof of Theorem
\ref{FR} given in \cite{FennRourke:1978}.  It follows from the ``if'' part
of Lemma \ref{lem:FR-Lemma9} and the fact that handle-slides and
stabilizations on a framed link $L$ preserve the homeomorphism class
of $M_L$ and the $\pi_1(W_L;p_1,p_k)$, $k=1,\ldots,n$.

Now we prove that (2) implies (1).
Assume that all the algebraic conditions are satisfied.
By Lemma~\ref{lem:FR-Lemma9}, we may assume,  after some
stabilizations, that 
 $W=\partial \Omega$, where $\Omega$ is a $5$-manifold such that
diagram~\eqref{diagram:FR-Lemma9} commutes and $j_\ast$ and $j'_\ast$
are split injections. Now we alter $\Omega$, as in the original proof
in \cite{FennRourke:1978}, by doing surgery on $\Omega$ until we have
$\pi_1(\Omega)\cong \pi_1(W_L)$. Then we modify $L$ and $L'$ to
$\tilde{L}$ and $\tilde{L'}$ by some specific 
stabilizations and handle-slides until we obtain a trivial cobordism $\Omega'$ joining
$W_{\tilde{L}}$ and $W_{\tilde{L'}}$. Thus $W_{\tilde{L}}$ and
$W_{\tilde{L'}}$ are two different relative handle decompositions of
the same manifold.

By a famous theorem of J. Cerf \cite{Cerf:stratification} any two relative handle decomposition of the same manifold are connected by a sequence of handle slides, creating/ annihilating canceling handle pairs and isotopies. (For a reference see \cite[Theorem 4.2.12]{Gompf-Stipsicz}.)
Note that Cerf's theorem applies in the case when $W_L$ has two boundary components, as well as in the case where the boundary of the $4$-manifold is connected.
Fenn and Rourke have shown in \cite{FennRourke:1978} that these handle
slides (1-handle slides and 2-handle slides) and creating or
annihilating canceling handle pairs can be achieved by modifying the
links using  stabilization and handle-slides.
Hence the proof is complete.  
\end{proof}

%%%%%%%%%%%%%%%%%%%%%%%%%%%%%%%%%%%%%%%%%%%%%%%%%%%%%%%%%%%%%%%%%%%%%%%%%%%%%%%%%%%%%%%%
\section{Null-homotopic framed links}\label{section: application1}

In this section we apply Theorem \ref{thm:FR-generalization} to
null-homotopic framed links.

Let $M$ be a compact, connected, oriented $3$-manifold with $n>0$
boundary components as before.  We use the notations given in
Section \ref{sec:gener-fenn-rourke-1}.

A framed link $L$ in $M$ is said to be {\em null-homotopic} if each
component of $L$ is null-homotopic in $M$.  In this case, the map
$$i_k\colon\pi_1(M;p_1,p_k)\to\pi_1(W_L;p_1,p_k)$$ is bijective for
$k=1,\dots,n$.  Define
$$e_k\colon\pi_1(M_L;p_1^L,p_k^L)\to\pi_1(M;p_1,p_k)$$ to be the composition
\begin{gather*}
  e_k\colon\pi_1(M_L;p_1^L,p_k^L)\overset{i'_k}{\to}
  \pi_1(W_L;p_1,p_k)\overset{i_k^{-1}}{\underset{\cong}{\to}}\pi_1(M;p_1,p_k),
\end{gather*}
which is surjective.

\begin{thm}
  \label{r4}
  Let $M$ be a compact, connected, oriented $3$-manifold with $n > 0$
  boundary components, and let $L,L'\subset M$ be null-homotopic framed links.  Then
  the following conditions are equivalent.
  \begin{enumerate}
  \item $L$ and $L'$ are related by a sequence of 
    stabilizations and handle-slides.
  \item There exists a homeomorphism $h\colon  M_{L}\to M_{L'}$ relative to
    the boundary such that the following diagram commutes for
    $k=1,\ldots,n$.
    \begin{gather} \label{diagram:r4}
      \xymatrix{
	\pi _1(M_L;p_1^L,p_k^L)\ar[rr]^{h_k}\ar[dr]_{e_k} &
	&
	\pi_1(M_{L'};p_1^{L'},p_k^{L'})\ar[ld]^{e'_k}\\
	&
	\pi_1(M;p_1,p_k)&
      }
    \end{gather}
  \end{enumerate}
\end{thm}

\begin{rem}
  \label{r5}
  For a closed $3$-manifold $M$, the variant of Theorem \ref{r4} is
  implicitly obtained in \cite{FennRourke:1978}.  Two null-homotopic
  framed links $L$ and $L'$ in a closed, connected, oriented
  $3$-manifold $M$ are related by a sequence of stabilizations and
  handle-slides if and only if there is a homeomorphism $h\colon
  M_{L}\to M_{L'}$ such that the diagram
    \begin{gather} \label{e5}
      \xymatrix{
	\pi _1(M_L)\ar[rr]^{h_*}\ar[dr]_{e} &
	&
	\pi_1(M_{L'})\ar[ld]^{e'}\\
	&
	\pi_1(M)&
      }
    \end{gather}
    commutes.  Here $e$ and $e'$ are defined similarly as before.
\end{rem}

Theorem \ref{r4} follows easily from Theorem \ref{thm:FR-generalization}
and the following lemma, which seems to be well known.  In fact, it
seems implicit in Fenn and Rourke \cite{FennRourke:1978}, p. 8,
ll. 8--9, where it reads ``For many other groups, $\eta(\Delta)$
vanishes, e.g. the fundamental group of any $3$-manifold.''  We give a
sketch of proof of this fact since we have not been able to find a
suitable reference.

\begin{lem}
  \label{r3}
  If $M$ is a compact, connected, oriented $3$-manifold, then we have
  $H_4(\pi_1M,\Z)=0$.
\end{lem}

\begin{proof}
  Consider a connected sum decomposition $M\cong M_1\sharp\dots\sharp
  M_k$, $k\ge0$, where each $M_i$ is prime.  Since
  $\pi_1M\cong\pi_1M_1*\dots*\pi_1M_k$, we have $H_4(\pi_1M,\Z)\cong
  H_4(\pi_1M_1,\Z)\oplus\dots\oplus H_4(\pi_1M_k,\Z)$.  Thus, we may assume
  without loss of generality that $M$ is prime.  If
  $M=S^2\times S^1$, then we have $H_4(\pi_1M,\Z)=H_4(\mathbb Z,\Z)=0$.
  Hence we may assume that $M$ is irreducible.

  If $\pi_1M$ is infinite, then $M$ is a $K(\pi_1M,1)$ space. Hence
  \begin{gather*}
    H_4(\pi_1M,\Z)\cong H_4(M,\Z)=0.  
  \end{gather*}

  Suppose that $\pi_1M$ is finite.  If $\partial M\neq\emptyset$, then
  we have $M\cong B^3$ and clearly $H_4(\pi_1M,\Z)=0$. Thus we may
  assume that $M$ is closed.  Then the universal cover of $M$ is a
  homotopy $3$-sphere, which is $S^3$ by the Poincar\'e conjecture
  established by Perelman.  By Lemma 6.2 of \cite{Adem-Milgram}, we
  have
  \begin{gather}
    \label{e2}
    H^5(\pi_1M,\Z)\cong H^1(\pi_1M,\Z).
  \end{gather}
  Recall that, for any finite group $G$,  $H_n(G,\Z)$ is finite for
  all $n\ge1$.  This fact and the universal coefficient theorem imply 
  \begin{gather}
    \label{e4}
    H^1(\pi_1M,\Z)\cong\operatorname{Hom}(H_1(\pi_1M,\Z),\Z)=0  ,\\
    \label{e3}
    H^5(\pi_1M,\Z)\cong \operatorname{Hom}(H_5(\pi_1M,\Z),\Z)\oplus\operatorname{Ext}(H_4(\pi_1M,\Z),\Z)
    \cong H_4(\pi_1M,\Z),
  \end{gather}
  where the last $\cong$ follows since $H_4(\pi_1M,\Z)$ is finite.
  Now, \eqref{e2}, \eqref{e4} and \eqref{e3} imply that $H_4(\pi_1M,\Z)=0$.
\end{proof}

%%%%%%%%%%%%%%%%%%%%%%%%%%%%%%%%%%%%%%%%%%%%%%%%%%%%%%%%%%%%%%%%%%%%%%%%%%%%%%%%%%%%%%%%
\section{$\pi_1$-admissible framed links} \label{section:admissible}
In this section we consider $\pi_1$-admissible framed links and give a refinement of Theorem~\ref{r4}. We also consider $\pi_1$-admissible framed links in cylinders over surfaces.

\subsection{$\pi_1$-admissible framed links in $3$-manifolds}
\label{sec:pi1-admiss-fram-1}
Let $M$ be a compact, connected, oriented $3$-manifold.
Let us call a framed link $L$ in $M$ {\em
$\pi_1$-admissible} if
\begin{itemize}
\item $L$ is null-homotopic, and
\item the linking matrix of $L$ is diagonal with diagonal entries
  $\pm1$, or, in other words, $L$ is algebraically split and
  $\pm1$-framed.
\end{itemize}

Surgery along $\pi_1$-admissible framed links has been studied by
Cochran, Gerges and Orr \cite{Cochran-Gerges-Orr}.  (They considered
there mainly more general framed links.)
In \cite{Cochran-Gerges-Orr} it is proved that, for all $d\ge1$, we
have $\pi_1(M_L)/\Gamma_d\pi_1(M_L)\cong\pi_1(M)/\Gamma_d\pi_1(M)$,
where for a group $G$, $\Gamma_dG$ denotes the $d$th lower central
series subgroup of $G$ defined by $\Gamma_1G=G$ and
$\Gamma_{d}G=[G,\Gamma_{d-1}G]$ for $d\ge 2$.  In this sense, surgery
along a $\pi_1$-admissible framed link $L$ in a $3$-manifold $M$ gives
a $3$-manifold $M_L$ whose fundamental group is very close to that of
$M$.

Surgery along $\pi_1$-admissible framed links was also studied by the
first author \cite{Habiro:2006}.  To state the result from
\cite{Habiro:2006} that we use in this section, we introduce
``band-slides'' and ``Hoste moves'', which are two special kinds of
moves on $\pi_1$-admissible framed links.

A {\em band-slide} is a pair of algebraically cancelling pair of
handle-slides of one component 
over another, see Figure \ref{F05}.
A band-slide on a $\pi_1$-admissible
framed link produces a $\pi_1$-admissible framed link.

A {\em Hoste move} is depicted in Figure \ref{FHoste}.  Let
$L=L_1\cup\cdots\cup L_l$ be a $\pi_1$-admissible framed link in $M$,
with an unknotted component 
$L_i$ with framing $\epsilon=\pm1$.  Since $L$ is $\pi_1$-admissible,
the linking number of $L_i$ and each component of $L':=L\setminus L_i$ is zero.  Let
$L'_{L_i}$ denote the framed link obtained from $L'$ by surgery along $L_i$,
which is regarded as a framed link in $M\cong M_{L_i}$.  The link
$L'_{L_i}$ is 
again $\pi_1$-admissible.  Then the framed links $L$ and $L'_{L_i}$ are
said to be related by a Hoste move.

\begin{figure}[t]
    \begin{center}\input{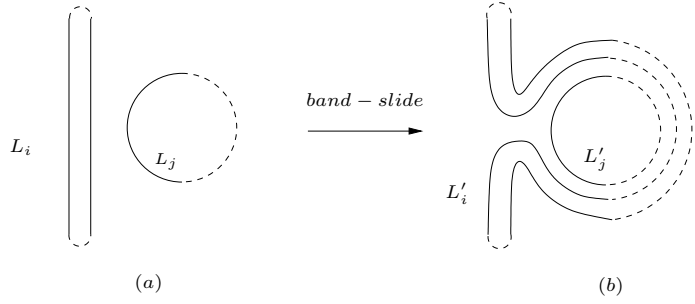}\end{center}
    \caption{(a) Two components $L_i$ and $L_j$ of a framed link.  (b) The result of a band-slide of $L_i$ over $L_j$.}
    \label{F05}
  \end{figure}

\begin{figure}[t]
    \begin{center}\input{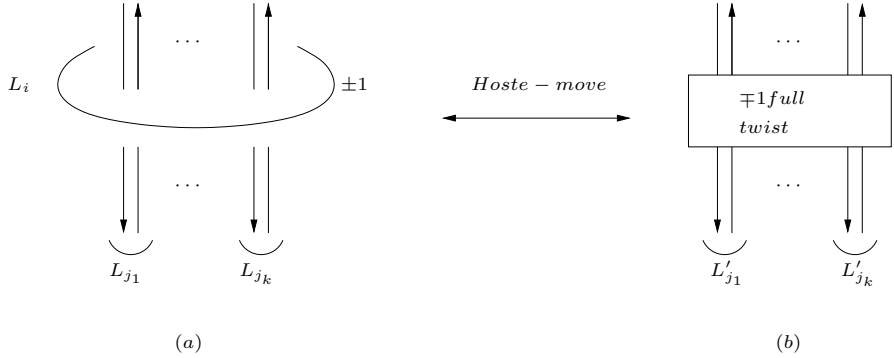}\end{center}
    \caption{(a) The component $L_i$ of $L$ is unknotted and of framing
  $\pm 1$.  (b) The result $L'_{L_i}$ of a Hoste move on  $L_i$.}
    \label{FHoste}
  \end{figure} 

\begin{prop}[{\cite[Proposition 6.1]{Habiro:2006}}]
  \label{prop:6}
  For two $\pi_1$-admissible framed links $L$ and $L'$ in a connected, oriented $3$-manifold $M$, the
  following conditions are equivalent.
  \begin{enumerate}
  \item \label{r11.1} $L$ and $L'$ are related by a sequence of stabilizations and handle-slides.
  \item \label{r11.2} $L$ and $L'$ are related by a sequence of stabilizations and band-slides.
  \item \label{r11.3} $L$ and $L'$ are related by a sequence of Hoste moves.
  \end{enumerate}
\end{prop}

Theorem \ref{r4} and Proposition \ref{prop:6} immediately imply the following result.

\begin{thm}
  \label{prop:admissible}
  Let $M$ be a compact, connected, oriented $3$-manifold with $n > 0$
  boundary components, and let $L,L'\subset M$ be $\pi_1$-admissible, framed links.  Then
  the following conditions are equivalent.
  \begin{enumerate}
  \item
    $L$ and $L'$ are related by a sequence of stabilizations and
    band-slides.
  \item $L$ and $L'$ are related by a sequence of Hoste moves.
  \item 
    There exists a homeomorphism $h\colon  M_{L}\to M_{L'}$ relative to the
    boundary such that the following diagram commutes for
    $k=1,\ldots,n$.
    \begin{gather} \label{e8}
      \xymatrix{
	\pi _1(M_L;p_1^L,p_k^L)\ar[rr]^{h_k}\ar[dr]_{e_k} &
	&
	\pi_1(M_{L'};p_1^{L'},p_k^{L'})\ar[ld]^{e'_k}\\
	&
	\pi_1(M;p_1,p_k)&
      }
    \end{gather}
  \end{enumerate}
\end{thm}

\subsection{$\pi_1$-admissible framed links in cylinders over
  surfaces}
\label{sec:pi1-admiss-fram}

In this subsection, we consider the special cases of Theorem
\ref{prop:admissible} where $M=\Sigma_{g,n}\times I$ is the cylinder
over a surface $\Sigma _{g,n}$ of genus $g\ge 0$ with $n\ge 0$
boundary components.  In this case, the condition (3) in Theorem
\ref{prop:admissible} can be weakened.

Let $L$ be a $\pi_1$-admissible framed link in the cylinder
$M=\Sigma_{g,n}\times I$.  By \cite[Theorem 6.1]{Cochran-Gerges-Orr}, there are
natural isomorphisms between nilpotent quotients
\begin{gather}
  \label{e6}
  \pi_1M_L/\Gamma_d\pi_1M_L\cong\pi_1M/\Gamma_d\pi_1M\cong \pi_1\Sigma_{g,n}/\Gamma _d\pi_1\Sigma_{g,n}.
\end{gather}
for all $d\ge1$.  

\subsubsection{Surfaces with nonempty boundary}
\label{sec:surf-with-bound-1}

Consider the case $n\ge 1$.  Note that $\partial
M=\partial(\Sigma_{g,n}\times I)$ is connected.

\begin{prop} \label{thm:nullhomotopic-1boundary-admissible}
Let $L$ and $L'$ be two $\pi_1$-admissible, framed links in
$M=\Sigma_{g,n} \times I$ with $n>0$.  Then the following conditions
are equivalent.
\begin{enumerate}
\item $L$ and $L'$ are related by a sequence of stabilizations and band-slides.
\item There exists a homeomorphism $h \colon  M_{L} \to M_{L'}$ relative to
the boundary.
\end{enumerate}
\end{prop}

\begin{proof}
``(1)$\Rightarrow$(2)'' immediately follows from Theorem~\ref{prop:admissible}.

To prove ``(2)$\Rightarrow$(1)'', one has to show that the diagram
\eqref{e8} commutes for $k=1$, i.e., 
\begin{gather} \label{e9}
  \xymatrix{
    \pi _1(M_L;p_1^L)\ar[rr]^{h_1}\ar[dr]_{e_1} &
    &
    \pi_1(M_{L'};p_1^{L'})\ar[ld]^{e'_1}\\
    &
    \pi_1(M;p_1)&
  }
\end{gather}
commutes.  This can be checked by using the isomorphism \eqref{e6}.
Let $x\in\pi_1(M_L;p_1^L)$.  For $d\ge 1$, take the nilpotent quotient of the diagram \eqref{e9}
\begin{gather} \label{e10}
  \xymatrix{
    \pi _1(M_L;p_1^L)\ar[rr]^{h_1}_{\cong}/\Gamma _d\ar[dr]_{e_1}^{\cong} &
    &
    \pi_1(M_{L'};p_1^{L'})/\Gamma _d\ar[ld]^{e'_1}_{\cong}\\
    &
    \pi_1(M;p_1)/\Gamma _d&
  }
\end{gather}
where all arrows are isomorphisms.  Since the homeomorphism $h\colon
M_L\congto M_{L'}$ respects the
boundary, the diagram \eqref{e10} commutes.  Hence, for $x\in\pi_1(M_L;p_1^L)$ we have
\begin{gather}
  \label{e11}
  e_1(x)\equiv e'_1h_1(x)\pmod{\Gamma _d\pi_1(M;p_1)}.
\end{gather}
Since \eqref{e11} holds for all $d\ge 1$, and since we have $\bigcap_{d\ge
 1}\Gamma _d\pi_1(M;p_1)=\{1\}$, it follows that $e_1(x)=e'_1h_1(x)$.
Hence the diagram \eqref{e9} commutes.
\end{proof}

\subsubsection{Closed surfaces}
\label{sec:closed-surfaces}
Now, we consider the case $n=0$. In this case, the manifold
$M=\Sigma_{g,0}\times I$ has two boundary components.
Set $F_1=\Sigma_{g,0}\times\{0\}$ and $F_2=\Sigma_{g,0}\times\{1\}$.
Choose a base point $p$ of $\Sigma_{g,0}$ and set
$p_1=(p,0)\in F_1$ and $p_2=(p,1)\in F_2$.

\begin{prop} \label{thm:nullhomotopic-SigmaxI-admissible}
Let $L$ and $L'$ be two $\pi_1$-admissible, framed links in
$M=\Sigma_{g,0} \times I$.  Then the following conditions are
equivalent:
\begin{enumerate}
\item $L$ and $L'$ are related by a sequence of stabilizations and band-slides,
\item There exists a homeomorphism $h\colon  M_{L} \to M_{L'}$ relative to the boundary such that the following diagram commutes:
\begin{equation}
\divide\dgARROWLENGTH by2
\begin{diagram}\label{diagram:thm:nullhomotopic-SigmaxI-admissible}
\node{\pi_1(M_L;p_1^{L},p_2^L)} \arrow{se,b}{e_2 } \arrow[2]{e,t}{h_2}
\node[2]{\pi_1(M_{L'};p_1^{L'},p_2^{L'})} \arrow{sw,r}{e'_2 }\\
\node[2]{\pi_1(M;p_1,p_2)}
\end{diagram}
\end{equation}
\end{enumerate}
\end{prop}

\begin{proof}
  The proof is similar to that of Proposition
  \ref{thm:nullhomotopic-1boundary-admissible}.  In the proof of
  ``(2)$\Rightarrow$(1)'', one has to prove that the diagram
  \eqref{e8} commutes for $k=1$.  This can be done similarly using the
  fact that 
  \begin{gather*}
    \bigcap_{d\ge 1}\Gamma _d\pi_1(M;p_1)=\bigcap_{d\ge 1}\Gamma _d\pi_1(\Sigma _{g,0};p_1)=\{1\}.
  \end{gather*}
\end{proof}

For the cylinder over the torus $T^2=\Sigma_{1,0}$, we do not need
commutativity of
\eqref{diagram:thm:nullhomotopic-SigmaxI-admissible} in
Proposition \ref{thm:nullhomotopic-SigmaxI-admissible}.

\begin{prop}
 Let $L$ and $L'$ be two $\pi_1$-admissible, framed links in the cylinder $M=T^2 \times I$.
Then the following conditions are equivalent.
\begin{enumerate}
\item $L$ and $L'$ are related by a sequence of stabilizations and band-slides.
\item There exists a homeomorphism $h \colon  M_{L} \to M_{L'}$ relative to the boundary.
\end{enumerate}
\end{prop}

\begin{proof}
By Proposition  \ref{thm:nullhomotopic-SigmaxI-admissible} we just
have to show that if there exists a homeomorphism $h\colon M_{L} \to
M_{L'}$ relative to the boundary, then there exists a homeomorphism $h'
\colon M_{L} \to M_{L'}$ such that the diagram
\eqref{diagram:thm:nullhomotopic-SigmaxI-admissible}, with $h_2$
replaced by $h_2'$, commutes.

Consider the cylinder $T^2 \times I$. Fix one boundary component while
twisting the other once along the meridian (resp. the longitude) of
$T^2$. This defines a self-homeomorphism $\tau_m$ (resp. $ \tau_l$) on
$T^2 \times I$ relative to the boundary which maps $\{\ast\}\times I$,
$\ast \in T^2$, to a line with the same endpoints but which travels
once along the meridian (resp. the longitude). A sequence of $\tau_m$
and $\tau_l$ defines a self-homeomorphism $s$ on $T^2 \times I$ by
using the composition of maps. Any bijective map $b\colon
\pi_1(T^2\times I;p_1,p_2) \to \pi_1(T^2\times I;p_1,p_2)$ of
$\pi_1(T^2\times I)$-sets can be
induced by such a self-homeomorphism. Let $M'_{L'}=M_{L'} \cup_{T^2} (
T^2 \times I)$ be a homeomorphic copy of $M_{L'}$ obtained by gluing
together $M_{L'}$ and $T^2\times I$ along $F_2\cong T^2 \subset
M_{L'}$ and $T^2\times \{0\} $ using the identity map.  Any
self-homeomorphism $s$ on $T^2 \times I$ as defined above, extends to
a self-homeomorphism $\tilde{s}$ on $M'_{L'}$. Thus, we can find a
self-homeomorphism $s$ on $T^2 \times I$ such that the composition
$h'=\tilde{s}\circ h$ defines a commutative diagram
\eqref{diagram:thm:nullhomotopic-SigmaxI-admissible}.

\end{proof}

\begin{rem}\label{rem1}
  If $g>1$, then the above proof can not be extended to the closed
 surface $\Sigma_{g,0}$. In this case, every self-homeomorphism  of
 $\Sigma_{g,0}$ is homotopic to the
 identity. This can be seen as follows. Every diffeomorphism $g\in
 \operatorname{Diff}(\Sigma_{g,0} \times I)$ relative to the boundary is
 homotopic to a diffeomorphism $g'(x,t):=(g_t(x),t)$ with $g_t(x) \in
 \operatorname{Diff}(\Sigma_{g,0})$. Since $g$ is the identity on the
 boundaries we have $g_0(x)=g_1(x)=\id_{\Sigma_{g,0}}(x)$. Hence, $g_t$
 defines a loop in $\operatorname{Diff}(\Sigma_{g,0})$ and every $g_t$ is
 homotopic to $\id_{\Sigma_{g,0}}$. Thus, $g_t$ is a loop in the group
 $\operatorname{Diff}_0(\Sigma_{g,0})$ of diffeomorphisms of $\Sigma_{g,0}$
 homotopic to the identity.  By a theorem of Earle and Eells
 \cite{Earle:1967} the group $\operatorname{Diff}_0(\Sigma_{g,0})$ is
 contractible when $g>1$.  Hence, the loop formed by
 $g_t$ is homotopic to $\id_{\Sigma_{g,0}}$ and therefore $g$ is homotopic to
 $\id_{\Sigma_{g,0} \times I}$.
\end{rem}

\section{Example}
\label{sec:example}

\subsection{An example}
Let us call the equivalence relation on framed links generated by
stabilizations and handle-slides the {\em $\delta$-equivalence}.

The following example shows that commutativity of diagram
\eqref{diagram:delta-enhanced} for $k=2,\ldots ,n$ is necessary as
well as that for $k=1$.
	
  Let $V_1$ and $V_2$ be handlebodies of genus $2$ and $1$,
  respectively, embedded in $S^3$ in a trivial way, and set
  $M=S^3\setminus \int(V_1\cup V_2)$, $F_k=\partial V_k$ ($k=1,2$), see Figure
  \ref{F01}(a).  \begin{figure}[t]
    \begin{center}\input{F01.pstex_t}\end{center}
    \caption{}
    \label{F01}
  \end{figure} Let $\beta ,\beta '\subset M$ be two arcs from
  $p_1\in F_1$ to $p_2\in F_2$, and let $a,b$ and $c$ be loops based at $p_1$,
  as depicted. The fundamental group $\pi_1M$ is freely generated by $a,b,c\in \pi_1M$.

  Let $L=L_1\cup L_2$ be the framed link in $M$ as depicted in Figure
  \ref{F01} (a), where $L_1$ and $L_2$ are of framing $0$.
  The result $M_L$ of surgery along $L$ is obtained from $M$ by
  letting the two handles in $V_1$ and $V_2$ clasp each other.
  $\pi_1M_L$ has a presentation $\langle a,b,c\;|\; aca^{-1}c^{-1}=1\rangle $.

  Let $f\colon M\congto M$ be a homeomorphism relative to the boundary
  such that $f(\beta ')=\beta $.  The image $f(L)=L'=L'_1\cup L'_2$
  looks as depicted in Figure \ref{F01}(b).  Let $h\colon M_L\congto
  M_{L'}$ be the homeomorphism induced by $f$.  Note that
  $\pi_1W_L\cong\langle b\rangle\cong\Z$ and $\pi_1W_{L'}\cong\langle
  b\rangle\cong\Z$.  Observe that diagram
  \eqref{diagram:delta-enhanced} is commutative for $k=1$ but not for
  $k=2$.  Hence Theorem \ref{thm:FR-generalization} can not be used
  here to deduce that $L$ and $L'$ are $\delta$-equivalent.

  In fact, $L$ and $L'$ are {\em not} $\delta$-equivalent.  We can
  verify this fact as follows.  Let $T$ be a tubular neighborhood of
  $\beta $ in $M$.  Let $K$ be a small $0$-framed unknot meridional to
  $T$.  Let $J$ be a knot in $\int V_1$, to which the loop $b$ is
  meridional, as depicted in Figure \ref{F04}(a), (b), and let $N(J)$
  denote a small tubular neighborhood of $J$ in $V_1$.  \begin{figure}[t]
    \begin{center}\input{F04.pstex_t}\end{center}
    \caption{}
    \label{F04}
  \end{figure} Set
  $M'=S^3\setminus \int N(J)$, which is homeomorphic to a solid torus.
  Let $K_1$ and $K_2$ be framed knots as depicted.  It suffices to
  prove that the framed links $\tilde L=L\cup K\cup K_1\cup K_2$ and
  $\tilde L'=L'\cup K\cup K_1\cup K_2$ in $M'$ are not
  $\delta$-equivalent.  Observe that $\tilde L$ (resp. $\tilde L'$) is
  $\delta$-equivalent to the $3$-component link depicted in Figure
  \ref{F04}(c) (resp. (d)).  (These links are the Borromean rings in
  $S^3$ with $0$-framings.)  One can show that these two links are not
  $\delta$-equivalent by using the invariant $B$ of framed links
  defined in Subsection \ref{aa} below.  For the framed links $L_c$
  and $L_d$ of Figure \ref{F04} (c) and (d), respectively, we have
  $B(L_c)=\{0\}$ and $B(L_d)=\Z$.

\subsection{An invariant of  $O_n$-$\pi_1$-admissible framed links in
  the exterior of an unknot in $S^3$}
\label{aa}
\newcommand\cups{\cup\cdots\cup}
\newcommand\bm{\bar\mu}
For $n\ge0$, let $O_n$ and $I_n$ denote the zero matrix and the
identity matrix, respectively, of size $n$.
For $p,q\ge0$, set $I_{p,q}=I_p\oplus(-I_q)$, where $\oplus$ denotes
block sum.

Let $J$ be an unknot in $S^3$ and set $E=S^3\setminus\int N(J)\cong
S^1\times D^2$, where $N(J)$ is a tubular neighborhood of $J$.  

Let $L=L^z_1\cups L^z_n\cup L^a_1\cups L^a_{p+q}$, $n,p,q\ge0$, be an
oriented, ordered, null-homotopic framed link in $E$ whose linking
matrix is of the form $O_n\oplus I_{p,q}$.  Let us call
such a framed link {\em $O_n$-$\pi_1$-admissible}.  Let us call
$L^z_1,\dots,L^z_n$ the {\em $z$-components} of $L$, and
$L^a_1,\dots,L^a_{p+q}$ the {\em
$a$-components} of $L$.

Since $L^z_1\cups L^z_n\cup J$ is algebraically
split, for $1\le i<j \le n$ the triple Milnor invariant
$\bar\mu(L^z_i,L^z_j,J)\in\Z$ is well defined.  Set
\begin{gather*}
  B(L)=\operatorname{Span}_{\Z}\{\bar\mu(L^z_i,L^z_j,J)\;|\;1\le i<j \le n\},
\end{gather*}
which is a subgroup of $\Z$.  Note that $B(L)$ does not depend on
the $a$-components of $L$.  Note also that $B(L)$ does not depend on the
ordering and orientations of the $z$-components of $L$.

\begin{lem}
  \label{r7}
  $B(L)$ is invariant under handle-slide of a $z$-component over 
  another $z$-component.
\end{lem}

\begin{proof}
  It suffices to consider a handle-slide of $L^z_1$ over $L^z_2$.  The
  link obtained from $L$ by this handle-slide is $L'=(L')^z_1\cup
  (L')^z_2\cups (L')^z_n\cup (L')^a_1\cups (L')^a_{p+q}$, where
  $(L')^z_1=L^z_1\sharp_b \tilde L^z_2$ is a band sum of $L^z_1$ and a
  parallel copy $\tilde L^z_2$ of $L^z_2$ along a band $b$, and
  $(L')^z_i=L^z_i$ for $i=2,\dots,n$.  We have
  \begin{gather*}
    \begin{split}
      \bm((L')^z_1,(L')^z_2,J)&=\bm(L^z_1,L^z_2,J),\\
    \bm((L')^z_1,(L')^z_i,J)&=\bm(L^z_1,L^z_i,J)+\bm(L^z_2,L^z_i,J)\quad
    (2\le i\le n),\\
    \bm((L')^z_i,(L')^z_j,J)&=\bm(L^z_i,L^z_j,J)\quad (2\le i<j\le n).
    \end{split}
  \end{gather*}
  Hence we have $B(L')=B(L)$.
\end{proof}

\begin{lem}
  \label{r6}
  $B(L)$ is invariant under band-slides.
\end{lem}

\begin{proof}
  Clearly, a band-slide of an $a$-component over another ($z$- or
  $a$-) component preserves $B$.  Lemma \ref{r7} implies that a
  band-slide of a $z$-component over another $z$-component preserves
  $B$.

  Consider a band-slide of a $z$-component $L^z_1$ of $L$ over an
  $a$-component $L^a_1$ of $L$.  Let $L'$ be the resulting link.  Let
  $L''$ denote the result from $L$ by the same band-slide as before, 
  but we use here the $0$-framing of $L^a_1$ for the band-slide.  By the
  previous case, it follows that $B(L'')=B(L)$.  The $z$-part
  $(L')^z(=(L')^z_1\cups (L')^z_n)$ of  $L'$ differs from the
  $z$-part $(L'')^z$ of $L''$ by self-crossing change of the component
  $(L')^z_1$.  Since the triple Milnor invariant is invariant under
  link homotopy, it follows that $B(L')=B(L'')$.  Hence $B(L)=B(L')$.
\end{proof}

\begin{prop}
  \label{r2}
  If two $O_n$-$\pi_1$-admissible framed links $L$ and $L'$ are
  $\delta$-equivalent, then we have $B(L)=B(L')$.
\end{prop}

\begin{proof}
  We give a sketch proof assuming familiarity with techniques on
  framed links developed in \cite{Habiro:2006}.

  If $L$ and $L'$ are $\delta$-equivalent, then after adding to $L$
  and $L'$ some unknotted $\pm1$-framed components by stabilizations,
  $L$ and $L'$ become related by a sequence of handle-slides.
  Clearly, stabilization on an $O_n$-$\pi_1$-admissible framed link
  preserves $B$.  So, we may assume that $L$ and $L'$ are related by a
  sequence of handle-slides.  It follows that $L$ and $L'$ have the
  same linking matrix $O_n\oplus I_{p,q}$, $n,p,q\ge0$.

  Recall that for each sequence $S$ of handle-slides between oriented,
  ordered framed links there is an associated invertible matrix
  $\varphi(S)$ with coefficients in $\Z$, see e.g. \cite{Habiro:2006}.
  In our case, a sequence from $L$ to $L'$ gives a matrix $P\in
  GL(n+p+q;\Z)$ such that
  \begin{gather}
    \label{e7}
    P (O_n\oplus I_{p,q}) P^t = (O_n\oplus I_{p,q}).
  \end{gather}
  (Here $P^t$ denotes the transpose of $P$.)
  Let $H_{n,p,q}<GL(n+p+q;\Z)$ denote the subgroup consisting of matrices
  satisfying \eqref{e7}.  It is easy to see that $H_{n,p,q}$ is
  generated by the following elements.
  \begin{itemize}
  \item[(a)] $Q\oplus I_{p+q}$, where $Q\in GL(n;\Z)$.
  \item[(b)] $\left(
    \begin{matrix}
      I_n &0\\X&I_{p+q}
    \end{matrix}
    \right)$, where $X\in \operatorname{Mat}_\Z(p+q,n)$.
  \item[(c)] $I_n\oplus R$, where $R\in O(p,q;\Z)=\{T\in
  GL(p+q;\Z)\;|\;TI_{p,q}T^t=I_{p,q}\}$.
  \end{itemize}
  Hence $\varphi(S)$ can be expressed as
  \begin{gather*}
    \varphi(S)=w_1^{\epsilon_1}\cdots w_k^{\epsilon_k},
  \end{gather*}
  where $k\ge0$, $\epsilon_1,\ldots,\epsilon_k\in\{\pm1\}$, and
  $w_1,\dots,w_k\in H_{n,p,q}$ are generators of the above form.
  
  By an argument similar to that in \cite{Habiro:2006}, we can show
  that there are framed links  $L^{(0)}=L,L^{(1)},\dots L^{(k)}=L''$ such that
  \begin{enumerate}
  \item for $i=1,\dots,k$, $L^{(k-1)}$ and $L^{(k)}$ are related by a
  sequence $S_i$ of handle-slides, orientation changes and permutations with
  associated matrix $\varphi(S_i)=w_i^{\epsilon_i}$,
  \item there is a sequence of band-slides from $L''$ and $L'$.
  \end{enumerate}
  Here the framed links $L^{(0)},\dots,L^{(k)}$ are
  $O_n$-$\pi_1$-admissible.

  Let $i=1,\dots,k$.  If $w_i$ is a generator of type (b) or (c), then
  we have $B(L^{(i-1)})=B(L^{(i)})$ since
  $S_i$ is a sequence of handle-slides of $a$-components over other
  ($z$- or $a$-)components.  If $w_i$ is a generator of type (a), then
  $S_i$ is a sequence of orientation changes of $z$-components,
  permutations of $z$-components, and handle-slides of $z$-components
  over $z$-components.  Clearly, orientation changes and permutations
  preserve $B$.  Handle-slides of $z$-components over $z$-components also preserve $B$ by
  Lemma \ref{r7}.

  By Lemma \ref{r6}, we have $B(L'')=B(L')$.

  Hence we have $B(L)=B(L')$.
\end{proof}

\bibliographystyle{plain} % sorted alphabetically, labeled with numbers

\end{document}